\documentclass[12pt,a4paper,final,DIV=13]{scrartcl}
\usepackage[utf8]{inputenc}
\usepackage[T1]{fontenc}
\usepackage{lmodern}

\usepackage[english]{babel}

\usepackage{doi}
\usepackage[numbers]{natbib}
\usepackage{etoolbox,letltxmacro,xifthen,ifdraft} % reinforce the environment
\usepackage[dvipsnames]{xcolor}
\usepackage{amsmath,amssymb,enumitem,bbm,mathrsfs,amsthm,aliascnt,mathtools}
\usepackage[retainorgcmds]{IEEEtrantools}

\usepackage{hyperref}
\usepackage[atend]{bookmark}

\hypersetup{
  colorlinks=true,
  linkcolor=[named]{RedViolet},
  citecolor=[named]{OliveGreen},
  filecolor=[named]{Periwinkle},
  urlcolor=[named]{magenta}
}

\setlist[enumerate,1]{label={(\roman*)}}
\setlist[enumerate,2]{label={(\alph*)}}
\setlist[enumerate,3]{label={(\Roman*)}}

\newcommand{\newsstheorem}[2]{
  \newaliascnt{#1}{dummy}
  \newtheorem{#1}[#1]{#2}
  \aliascntresetthe{#1}
  \expandafter\def\csname #1autorefname\endcsname{#2}
}
\theoremstyle{plain}
  \newsstheorem{theorem}{Theorem}
  \newsstheorem{proposition}{Proposition}
  \newsstheorem{corollary}{Corollary}
  \newsstheorem{lemma}{Lemma}
\theoremstyle{definition}
  \newsstheorem{definition}{Definition}
  \newsstheorem{example}{Example}
  \newsstheorem{assumption}{Assumption}
\theoremstyle{remark}
  \newsstheorem{remark}{Remark}

\makeatletter
  \thm@headfont{\bfseries\sffamily}
\makeatother

\makeatletter
\newenvironment{proofthm}[1]{\par
  \pushQED{\qed}%
  \normalfont \topsep6\p@\@plus6\p@\relax
  \trivlist
  \item[\hskip\labelsep
        \bfseries\sffamily
    #1\@addpunct{.}]\ignorespaces
}{%
  \popQED\endtrivlist\@endpefalse
}
\makeatother

\newenvironment{eqnarr}{\begin{IEEEeqnarray}{rCl}}{\end{IEEEeqnarray}\ignorespacesafterend}
\newenvironment{eqnarr*}{\begin{IEEEeqnarray*}{rCl}}{\end{IEEEeqnarray*}\ignorespacesafterend}

\IEEEeqnarraydefcolsep{9}{4em}

\providecommand{\email}[1]{\href{mailto:#1}{\nolinkurl{#1}}}

\newcommand\RR{\mathbb{R}}
\newcommand\Zz{\mathbf{Z}}
\newcommand\Pp{\mathcal{P}}

\newcommand\PP{\mathbb{P}}
\newcommand\EE{\mathbb{E}}
\newcommand\Indic[1]{\Ind_{\{#1\}}}

\newcommand\Ind{\mathbbm{1}}
\newcommand\from{\colon}

\newcommand\NN{\mathbb{N}}

\newcommand\FF{\mathscr{F}}
\newcommand\Uu{\mathcal{U}}
\newcommand\Gg{\mathcal{G}}

\newcommand\Ee{\mathcal{E}}
\newcommand\pp{\mathbf{p}}
\newcommand\e{\mathrm{e}}
\newcommand\Hh{\mathcal{H}}

\newcommand\Zzu{\Zz^{(u)}}

\newcommand{\crochet}[1]{{\langle #1 \rangle}}

\makeatletter \newcommand\mathof[1]{{\operator@font#1}} \makeatother

\newcommand\dd{\mathof{d}}

\DeclareMathOperator{\argmin}{argmin}
\DeclarePairedDelimiter\abs{\lvert}{\rvert}

\DeclarePairedDelimiterX\ip[2]{\langle}{\rangle}{#1,#2}

\title{The strong Malthusian behavior of growth-fragmentation processes}
\author{Jean Bertoin\footnote{Institute of Mathematics, University of Zurich, Switzerland}  \and Alexander R. Watson\footnote{School of Mathematics, University of Manchester, UK}}
\date{\textsc{Draft}, \today}

\begin{document}

\maketitle 

\begin{abstract}
  Growth-fragmentation processes describe the evolution of systems
  of cells which grow continuously and fragment suddenly; they 
  are used in models of cell division and protein polymerisation.
  Typically, we may expect that in the long run, the concentrations of cells with given masses
  increase at some exponential rate, and that, after compensating for this,
  they arrive at an asymptotic profile.
  Up to now, this question has mainly been studied for the average
  behavior of the system, often by means of a natural
  partial integro-differential equation and the associated spectral
  theory. However, the behavior of the system as a whole, rather
  than only its average, is more delicate.
  In this work, we show that a criterion found by one of the authors
  for exponential ergodicity on average is actually sufficient to
  deduce stronger results about the convergence of the entire
  collection of cells to a certain asymptotic profile,
  and we find some improved explicit conditions for this to occur.
\end{abstract}

{\small 
}

\section{Introduction}
This work is concerned with the large time asymptotic behavior of a class of branching Markov processes in continuous time, which we call {\em growth-fragmentation processes}. These may be used  to model the evolution of a population, for instance of bacteria,
in which an individual reproduces by fission into two or more new
individuals.

Each individual
grows continuously, with the growth depending
deterministically on the current mass of the individual, up to
a random instant at which fission occurs.
This individual, which may be thought of as a mother,
is then replaced by a family of new individuals, referred to as her
daughters. We assume that mass is preserved at fission, meaning
that the mass of the mother immediately before the division
is equal to the sum of the masses of her daughters immediately afterwards.
The time at which the fission occurs and the masses of her daughters
at fission are both random, and depend on the mass of the mother
individual.
After a fission event, the daughters are in turn viewed as mothers of
future generations,
and evolve according to the same dynamics,
independently of the other individuals.

Mathematically, we represent this as
a process in continuous time, $\Zz=(\Zz_t, t\geq 0)$,
with values in the space of point measures on $(0,\infty)$.
Each individual is represented as an atom in $\Zz_t$,
whose location is the individual's mass.
That is, if at time $t$ there are $n\in \NN \cup\{\infty\}$ individuals
present, with masses $z_1,z_2,\dotsc$, then
$\Zz_t = \sum_{i=1}^n \delta_{z_i}$, with $\delta_z$ the Dirac delta
at $z \in (0,\infty)$.

Growth-fragmentation processes are members of the family of structured population models, which were first studied using analytic methods in the framework of linear integro-differential equations. To demonstrate this connection,
consider the \emph{intensity measure} $\mu_t$ of $\Zz_t$, defined by 
$\crochet{\mu_t,f} = \EE[\crochet{ \Zz_t,f}]$ for all $f\in{\mathcal C}_c$. That is, $f$ is a continuous function on $(0,\infty)$ with compact support, and  the notation $\crochet{m,f} = \int f \, \dd m$ is used for the integral of a function $f$ against a Radon measure $m$ on $(0,\infty)$, whenever this makes sense.
In words, $\mu_t(A)$  describes the concentration of individuals at time $t$ with masses in the set $A\subset (0,\infty)$, and, informally, the evolution of the branching Markov process $\Zz$ entails that 
the family $(\mu_t)_{t\geq 0}$ solves an evolution equation
(see \cite{EN00} for background) of the form
\begin{equation}\label{E:GFE}
\frac{\dd }{\dd t} \crochet{\mu_t ,f}=  \crochet{ \mu_t, {\mathcal A} f}, \qquad f\in{\mathcal C}_c^1,
\end{equation}
where the infinitesimal generator
\[
  {\mathcal A}f(x)
  =
  c(x) f'(x) + {B}(x) \int_{\Pp} \left( \sum_{i=1}^{\infty}f(xp_i)-f(x)\right)\kappa(x, \dd \pp)
\]
is naturally associated to the dynamics of $\Zz$
and $f$ is a smooth function in the domain of ${\mathcal A}$.
The meaning of this operator will be described precisely later,
when we derive it in equation \eqref{E:GFG}.
Briefly, $c\colon (0,\infty)\to (0,\infty)$ is a continuous function
representing the growth rate, $B\colon (0,\infty) \to [0,\infty)$ is a
bounded measurable function representing the fission rate,
and $\kappa$ a measurable probability kernel describing
the relative masses of the daughters obtained at fission.
That is, an individual of mass $x$ grows at rate $c(x)$,
experiences fission at rate $B(x)$ and, if fission occurs,
then the relative masses of the daughters are drawn from
the distribution $\kappa(x,\cdot)$.
We shall refer to \eqref{E:GFE}  as the {\em growth-fragmentation equation}. 

A fundamental problem in this analytic setting is to determine explicit conditions on the parameters governing the evolution of the system that ensure the so-called (asynchronous)  {\em Malthusian behavior}:    for all $f\in{\mathcal C}_c$, 
\begin{equation} \label{E:Malthus}
 \EE[\crochet{ \Zz_t,f}]=\crochet{\mu_t,f} \sim  \e^{\lambda t}\crochet{\mu_0,h } \crochet{\nu,f}\qquad \text{as }t \to \infty,
 \end{equation}
where  $\lambda \in \RR$, $h$ is positive function, and $\nu$ a Radon measure  on $(0,\infty)$ with $\crochet{\nu,h }=1$. 
When \eqref{E:Malthus} holds, we call $\lambda$ the
\emph{Malthus exponent} and $\nu$ the \emph{asymptotic profile}.
There exists a vast literature on this topic, and we content ourselves here to cite  a few  contributions  \cite{BerGab2, CalvoDoumicPerthame, DDGW, Esco} amongst the most recent ones, in which many further references can be found.

  Spectral analysis of the infinitesimal generator ${\mathcal A}$  often plays a key role for establishing \eqref{E:Malthus}. Indeed, if 
 there exist $\lambda \in \RR$, a positive function $h$ and a Radon measure $\nu$ that solve the eigenproblem
 \begin{equation} \label{E:eigenp}
 {\mathcal A}h =\lambda h \ ,\   {\mathcal A}'\nu =\lambda \nu\ ,\
 \crochet{\nu,h }=1,
 \end{equation}
 with $\mathcal{A}'$ the adjoint operator to $\mathcal{A}$,
 then \eqref{E:Malthus} follows rather directly.  In this direction, the Perron-Frobenius paradigm, and more specifically 
the Krein-Rutman theorem (which requires compactness of certain operators related to ${\mathcal A}$) yield
a powerful framework for establishing the existence of solutions to the eigenproblem \eqref{E:eigenp}. This method
has been widely used in the literature; see, for instance, \cite{Per07, BCG-fine, DoumGab, MS16}.
Then $\lambda$ arises as the leading eigenvalue of ${\mathcal A}$, i.e., the eigenvalue with the maximal real part, and 
$h$ and $\nu$ respectively as a corresponding positive eigenfunction and dual eigenmeasure.

A stochastic approach  for establishing \eqref{E:Malthus}, which is based on the Feynman-Kac formula and circumvents spectral theory, has been
developed by the authors in \cite{BW, BeFK} and Cavalli in \cite{Cavalli}.
To carry out this programme, we introduce, under the assumption
$  \sup_{x>0} c(x)/x<\infty$,  the unique
strong Markov process $X$ on $(0,\infty)$ with generator
\[ \mathcal{G} f (x)
  = \frac{1}{x} \mathcal{A}\bar{f}(x) - \frac{c(x)}{x} f(x), \]
where $\bar{f}(x) = x f(x)$. 
Assume that $X$ is irreducible and aperiodic, and 
define the Feynman-Kac weight
$$\mathcal{E}_t = \exp\left(\int_0^t \frac{c(X_s)}{X_s} \, \dd s\right),$$
and the Laplace transform 
$$L_{x,y}(q) 
= \EE_x[ \e^{- qH(y)} \mathcal{E}_{H(y)} \Indic{H(y)<\infty}], $$
where $H(y)=\inf\{t>0: X_t=y\}$ denotes the first hitting time of $y$ by $X$.
A weaker version of Theorem 1.2 in \cite{BeFK}  (see also Theorem 1.1 in \cite{BW}) can then be stated as follows.
\setcounter{dummy}{-1} % Start counters from 0.
\begin{theorem}\label{t:0}
  Fix $x_0 > 0$. Define
  \[ \lambda = \inf\{q\in \RR: L_{x_0,x_0}(q) < 1\}. \]
  The value of $\lambda \in\RR$ and is independent of $x_0$.
  If
  \begin{equation} \label{E:main}
    \limsup _{x\to 0+} \frac{c(x)}{x}< \lambda \quad\text{and}\quad  \limsup _{x\to \infty} \frac{c(x)}{x}<\lambda,
  \end{equation}
  then the Malthusian behavior \eqref{E:Malthus} holds (so $\lambda$ is the Malthus exponent)
  with 
  $$h(x) = xL_{x,x_0}(\lambda) \quad \text{and} \quad
  \nu(\dd y) = \frac{\dd y}{h(y)c(y) \lvert L'_{y,y}(\lambda)\rvert}.$$
\end{theorem}

Indeed, in \cite{BeFK}, it was even shown that \eqref{E:main} implies
that \eqref{E:Malthus} occurs at exponential rate.
\autoref{t:0} will form the basis of our work,
the purpose of which is to investigate the analog of \eqref{E:Malthus} for the random variable 
$\crochet{ \Zz_t,f}$ itself, rather than merely its expectation.
More precisely, assuming for simplicity that  the growth-fragmentation  process $\Zz$ starts from a single individual with mass $x>0$ and writing $\PP_x$ for the corresponding probability law, we prove the following result:
\begin{theorem}\label{t:main}
  Under the assumptions of \autoref{t:0},
  the process $\Zz$ exhibits \emph{strong Malthusian behavior}: for all $x>0$ and for $f$ any continuous function satisfying $\lVert f/h \rVert_{\infty}<\infty$, one has
    \begin{equation} \label{E:MalthusF}
    \lim_{t\to \infty} \e^{-\lambda t} \crochet{ \Zz_t,f}=   \crochet{\nu,f} W_{\infty} \qquad \text{ in }L^1(\PP_x),
  \end{equation}
  where 
  $$W_\infty=\lim_{t\to \infty} \e^{-\lambda t} \crochet{ \Zz_t,h} \quad\text{and} \quad \EE_x[W_{\infty}]=h(x).$$
\end{theorem}

%Interestingly, the condition \eqref{E:main} was shown in \cite{BeFK} to imply that the Malthusian behavior \eqref{E:Malthus} occurs exponentially fast.
The criterion \eqref{E:main} involves the Malthus exponent $\lambda$, which is itself usually not explicitly known. It might therefore appear unsatisfactory.
However, one can easily obtain lower-bounds for $\lambda$ solely in terms of the characteristics of the growth-fragmentation process, and these yield a fully explicit criterion.
We give an example of such a result as a conclusion to this work.

Of course, even though the Malthusian behavior \eqref{E:Malthus} suggests that its strong version \eqref{E:MalthusF} might hold, this is by no means automatic. For instance, it should be plain that \eqref{E:MalthusF} cannot hold when  $\lambda$ is negative.

The question of strong Malthusian behavior
has been considered  in the literature on branching processes for several different models, including
general Crump-Mode-Jagers branching processes \cite{Nerman, Jagers,JN84},
branching random walks \cite{Biggins92},
branching diffusions \cite{BBHHR,  EnHaKy, GiHaHa, HaHa, HaHeKy}, 
branching Markov processes \cite{AsHer, CRY, ChYu, Shio}, 
pure fragmentation processes \cite{Beresty, BeRou, Ber-frag}
and certain other growth-fragmentation processes \cite{BBCK, Dadoun, ShiQ-ou}.
A notable recent development is the study of the neutron transport equation
and associated stochastic processes
\cite{CHHK, HKV, HHK}, which uses a different probabilistic approach based on the notion of quasi-stationarity, as in \cite{ChVil}.
%\textcolor{red}{[** I moved all refs to neutron transport down here, even though only \cite{HHK} deals with the \emph{strong} Malthusian behavior.]}
Of course, these are just a sample of works on this topic,
and many more references can be found cited within them.
In particular, we can view $(\Zz_t, t\ge 0)$ as a general branching process
in the sense of Jagers \cite{Jagers}. This means that, rather than tracking
the mass of individuals at a given time, we instead track the birth time, birth
mass and death (i.e., fission) time
of every individual
in each successive generation. This process can be characterised in terms
of a \emph{reproduction kernel}; given the birth time and mass of an
individual, this describes the distribution of the
birth times and masses of its daughters.
Assuming that this general branching
process is Malthusian and supercritical (as defined in Section 5 of \cite{Jagers} in terms of the reproduction kernel), and that a certain $x\log x$  integrability condition and some further technical assumptions are fulfilled, Theorem 7.3 in \cite{Jagers} essentially states that
\eqref{E:MalthusF} holds with $W_{\infty}$ the terminal value of the  so-called {\em intrinsic martingale}. However, the assumptions and the quantities appearing in Theorem 7.3 in \cite{Jagers} are defined in terms of the reproduction kernel, sometimes in an implicit way. It appears to be rather difficult to understand the hypotheses and conclusions of \cite{Jagers} in terms of the parameters of the growth-fragmentation process; for instance, it does not seem to be straightforward to connect the general branching
process with the eigenproblem \eqref{E:eigenp}.

Our approach combines classical elements with some more recent ingredients.
%Our previous work \cite{BW,BeFK} can be used to show that the
%the Malthusian behavior \eqref{E:Malthus} holds for the intensity measure.
Given the Malthusian behaviour recalled in \autoref{t:0},
the main technical issue is 
to find explicit conditions, in terms of the characteristics of the growth-fragmentation, which ensure the uniform integrability of the intrinsic martingale.
%\footnote{\textcolor{red}{Removed this sentence: As alluded to above, growth-fragmentation processes belong to the
%family of general branching processes (recall that growth-fragmentations evolve in continuous time, whereas many models on branching processes rather concern discrete times).}}
However, the intrinsic martingale is defined in terms of the generations of the
associated general branching process rather than in natural time (see Section 5 of \cite{Jagers}), and it is difficult to connect this to the dynamics of the growth-fragmentation process.
%, which are essentially given by its infinitesimal generator $\mathcal{A}$.

We will circumvent this difficulty as follows.
As \autoref{t:0} may suggest, we first establish a so-called many-to-one
(or Feynman-Kac) formula, which provides an expression for the intensity measure $\mu_t$ of the point process $\Zz_t$ in terms of a functional of the (piecewise deterministic) Markov process $X$. Making use of results in \cite{BW}, this enables us to confirm that $\mu_t$ indeed solves the growth-fragmentation equation \eqref{E:GFE},
%(which should certainly not come as a surprise since the growth-fragmentation process and the growth-fragmentation equation are indeed meant to  model two sides of the same phenomenon)
and to construct a remarkable additive martingale associated with the growth-fragmentation process $\Zz$, namely
$$W_t=\e^{-\lambda t} \crochet{\Zz_t, h},\qquad t\geq 0,  $$
where the Malthus exponent $\lambda$ and the function $h$ are defined in terms of the Markov process $X$.   
In fact, $W$ is nothing but the version in natural times of the intrinsic martingale indexed by generations, as defined in Section 5 of \cite{Jagers}.  
We shall then prove that the boundedness  in $L^2(\PP_x)$, and hence the uniform integrability, of the martingale $W$ follows from \eqref{E:main} by adapting the well-known spinal decomposition technique
(described in \cite{BigKyp} for branching random walks) to our framework.
      
The spine process, which is naturally associated to the intrinsic martingale, plays an important role in the proof of  the strong Malthusian behavior \eqref{E:MalthusF}. Specifically, it yields a key tightness property for the random point measures $\Zz_t$, which then enables us to focus on individuals with masses bounded away from $0$ and from $\infty$. This is crucial to extend the original method of Nerman \cite{Nerman} to our setting. 
     
The rest of this paper is organized as follows. In Section 2, we describe the precise construction of the growth-fragmentation process $\Zz$, which is needed in Section 3 to establish a useful many-to-one formula for the intensity measure $\mu_t$ of $\Zz_t$. In particular, a comparison with results in \cite{BW} makes  the connection with the growth-fragmentation equation \eqref{E:GFE} rigorous. 
The $L^2$-boundedness of the intrinsic martingale is established in Section 4 under the assumption \eqref{E:main},
and we then prove the strong Malthusian behavior \eqref{E:MalthusF} in Section 5. Section 6 is devoted to providing explicit conditions on the characteristics of the growth-fragmentation that ensure \eqref{E:main}. 
     
\section{Construction of the growth-fragmentation process}
To start with, we introduce the three characteristics $c, {B}$ and $\kappa$ which govern the dynamics of the growth-fragmentation process. First, let   $c\from (0,\infty) \to (0,\infty)$ be a continuous function with 
\begin{equation}\label{e:condc}
\sup_{x>0} c(x)/x<\infty,
\end{equation}
which describes the growth rate of individuals as a function of their masses. 
For every $x_0>0$, the initial value problem
\begin{equation}\left\{
  \begin{aligned}
    \dot{x}(t) &= c(x(t)), \qquad t \geq  0, \\
    x(0) &= x_0,
  \end{aligned}
  \right.
  \label{e:ode}
\end{equation}
has a unique solution that we  interpret as the mass at time $t$ of an individual with initial mass $x_0$ when no fission  occurred before time $t$. 

Next, we consider a bounded,
measurable function ${B}\colon (0,\infty)\to [0,\infty)$,
% such that $\sup_{(0,K]}B<\infty$ for every $K>0$,
which specifies the rate at which a particle breaks (or branches) as a function of its mass. That is,  the probability that no fission event has occurred by time $t>0$ when the mass at the initial time is $x_0$, is given by
$$\PP_{x_0}[\text{no fission before time }t]=\exp\left(-\int_0^t {B}(x(s))\dd s\right)=\exp\left(-\int_{x_0}^{x(t)} \frac{{B}(y)}{c(y)}\dd y\right).$$

To complete the description and specify the statistics at fission events, we need to introduce some further notation. We  call a non-increasing sequence $\pp=(p_1, p_2, \ldots)$ in the unit sphere of $\ell^1$, i.e.,
$$  p_1\geq p_2\geq \dotsb \geq 0 \text{ and } \sum_{i\geq 1}p_i= 1,$$
a (proper) {\em mass partition}.  In our setting, we interpret a mass partition as the sequence (ranked in the non-increasing order) of the daughter-to-mother mass ratios at a fission event, agreeing that $p_i=0$ when the mother begets less than $i$ daughters. 
The space $\Pp$ of mass partitions is naturally endowed with the $\ell^1$-distance
and we write ${\mathcal B}(\Pp)$ for its Borel $\sigma$-algebra. We consider a probability kernel 
$$\kappa  \from (0,\infty)\times {\mathcal B}(\Pp) \to [0,1],$$
and think of $\kappa (x,\dd \pp)$ as the distribution of the random mass partition resulting from a fission event that occurs when the mother has mass $x>0$. We always implicitly assume that $\kappa(x, \dd \pp)$ has no atom at the trivial mass partition $(1,0,0, \ldots)$, as the latter corresponds to a fictive fission.

We next provide some details on the construction of growth-fragmentation processes and make the framework rigorous.  We denote by $\Uu = \bigcup_{n\ge 0} \NN^n$ the Ulam-Harris tree of finite sequences of positive integers, which will
serve as labels for the individuals.
As usual, we interpret the length $|u|=n$ of a sequence $u\in\NN^n$ as a generation, and for $i\in\NN$, 
write $ui$ for the sequence in $\NN^{n+1}$ obtained by aggregating $i$ to $u$ as its $(n+1)$-th element,  viewing then $ui$ as the $i$-th daughter of $u$.
The unique element of $\NN^0$, written $\varnothing$, will represent
an initial individual. 

We fix $x_0>0$ and aim at constructing the growth-fragmentation process $(\Zz_t, t\geq 0)$ started from 
a single atom at $x_0$, which we understand to represent a single
progenitor individual, \emph{Eve}. We denote by $\PP_{x_0}$ the corresponding
probability measure.
First consider a random variable
$\zeta $ in $(0, \infty]$ with cumulative distribution function
$$\PP_{x_0}[\zeta  \leq t]= 1-\exp\left(-\int_{x_0}^{x(t)} \frac{{B}(y)}{c(y)}\dd y\right), \qquad t\geq 0,$$
where $x(\cdot)$ denotes the solution to the flow velocity \eqref{e:ode} started from $x_0$. We view $\zeta $ as the fission time of Eve, and thus the trajectory of Eve is 
$$Z_t^{\varnothing}=x(t) \text{ for }t<\zeta .$$
We further set $b^{\varnothing}= 0$ and $d^{\varnothing}=\zeta $, so $[b^{\varnothing}, d^{\varnothing})$ is the time interval during which Eve is alive.  We also view 
$d^{\varnothing}$ as the birth-time of the daughters of Eve and thus set  $b^i=d^{\varnothing}$ for every $i\in\NN$.

Next, conditionally on $d^{\varnothing}=s<\infty$, that is, equivalently, on $Z_{d^{\varnothing}-}^{\varnothing}=x$
with $x=x(s)$, we pick a random mass partition $\pp=(p_1, \ldots)$ according to the law $\kappa(x,\dd \pp)$. We view 
$xp_1, xp_2, \ldots$ as the masses at birth of the daughters of Eve and continue the construction iteratively in an obvious way. That is, conditionally on  $xp_i=y>0$,  the lifetime $\zeta^i$ of the $i$-th daughter of Eve has the same distribution as $\zeta $ under $\PP_{y}$. 
Further set $d^i=b^i+\zeta^i$, and the trajectory of the $i$-th daughter of Eve is thus
$$Z_{t}^{i}=x(t-b^i) \text{ for }t\in[b^i,d^i),$$
with $x(\cdot)$ now denoting the solution to  \eqref{e:ode} started from $y$.
We stress that, thanks to \eqref{e:condc}, the boundary point $0$ is a trap for the flow velocity, in the sense that  
the solution to \eqref{e:ode} with initial value $x(0)=0$ is $x(t)=0$ for all $t$. Thus $0$ serves a cemetery state for particles, and individuals with zero mass can be simply discarded. 

 This enables us construct recursively a trajectory $(Z^u_t: t\in[b^u, d^u))$ for every $u\in \Uu$, and  the state of the growth-fragmentation at time $t$ is then given by the point measure on $(0,\infty)$ with atoms at the locations of the individuals alive at time $t$, viz.
$$\Zz_t = \sum_{u\in\Uu} \Ind_{\{t\in[b^u, d^u)\}} \delta_{Z^u_t}.$$
We stress that the number of individuals may explode at a finite time even in situations when every mother always begets finitely many children  (see, e.g. \cite{Savits}), and then infinitely many fission events may occur on any non-degenerate time interval. On the other hand, it is readily seen from
our key assumption \eqref{e:condc}
that the total mass process
increases at most exponentially fast, specifically 
$$\crochet{\Zz_t, {\mathrm {Id}}}\leq x\e^{\gamma t}, \qquad \PP_x{\text{-a.s.}} $$
where $\gamma=\sup_{x>0} c(x)/x$. Thus the point process $\Zz_t$ is always locally finite; however the growth-fragmentation is not always a continuous time Markov chain.

\section{A many-to-one formula}

The first cornerstone of our analysis is a useful  expression for the expectation of the integral of some function with respect to the random point measure $\Zz_t$ in terms of a certain Markov process $X$ on $(0,\infty)$.  In the literature, such identities are commonly referred to as
many-to-one formulas, they go back to \cite{KaP, Pey} and are known to play a crucial role in the analysis of branching processes.

Recall that  a size-biased pick from  a mass partition $\pp=(p_1, \ldots)$ refers to a random element 
$p_K$, where the distribution of the random index $K$ is $\PP(K=i)=p_i$ for $i\in\NN$. Size-biased picking enables us to map the probability kernel $\kappa$ on $(0,\infty)\times \Pp$
into a kernel $\bar k$ on $(0,\infty)\times (0,1)$ by setting for every $x>0$
$$\int_{(0,1)} g(r)\bar k(x,\dd r)= {B}(x) \int_{\Pp} \sum_{i=1}^{\infty} p_ig(p_i) \kappa(x, \dd \pp)$$
for a generic measurable function $g\from (0,1)\to \RR_+$. 
We stress that $\int_{(0,1)} \bar k(x,\dd r)= {B}(x)$ since $\kappa$ is a probability kernel on the space of proper mass partitions. 
We then introduce the operator 
$${\mathcal G} f (x) = c(x) f'(x) +\int_{(0,1)} (f(rx) - f(x)) \bar k(x, \dd r),$$
say defined for functions $f: (0,\infty)\to \RR$ which are bounded and possess a bounded and continuous derivative. It is easily seen that  ${\mathcal G}$
is the infinitesimal generator of a unique Markov process, say $X=(X_t, t\geq 0)$. 
Recall that we have assumed condition \eqref{e:condc}
and that $B$ is bounded.
% Since the function determined by \eqref{e:ode}
% cannot blow up in finite time, the fact that $B$ may be unbounded
% at $+\infty$ does not impede the construction of $X$.
By a slight abuse, we also use also the notation $\PP_{x_0}$ for the probability measure under which this piecewise deterministic Markov process starts from $X_0=x_0$.

The evolution of $X$ can be described in words as follows. The process is driven by the flow velocity \eqref{e:ode} until it makes a first downwards jumps; more precisely, the total rate of jump at state $x$ is $\int_{(0,1)}\bar k(x,\dd r)={B}(x)$. Further, conditionally on the event that a jump  occurs when the process is about to reach $x$, the position after the jump is distributed according to the image of the probability law ${B}(x)^{-1}\bar k(x, \dd r)$ by the dilation $r\mapsto rx$. An alternative formulation which makes the connection to the growth-fragmentation process more transparent, is that  $X$ follows the path of Eve up to its fission, then picks a daughter at random according to a size-biased sampling and follows her path, and so on, and so forth. 

We now state a useful representation of the intensity measure of $\Zz_t$ in terms of the Markov process $X$.

\begin{lemma}[Many-to-one formula -- Feynman-Kac representation] \label{L1}
  Define, for every $t\geq 0$,
  \[ 
    \Ee_t = \exp\biggl\{\int_0^t \frac{c(X_s)}{X_s} \, \dd s\biggr\}.
  \]
  For every measurable $f\from (0,\infty)\to \RR_+$ and every $x_0>0$, we have
  \[
    \EE_{x_0} \left[ \langle \Zz_t, f\rangle \right]
    = x_0\EE_{x_0}\left[ \frac{f(X_t)}{X_t}\Ee_t \right].
  \]
\end{lemma}
Lemma \ref{L1} is closely related to Lemma 2.2 in \cite{BW}, which provides a representation of the solution to the growth-fragmentation equation \eqref{E:GFE} by Feynman-Kac formula. Specifically, introduce the growth-fragmentation  operator ${\mathcal A}$
given for every $f\in{\mathcal C}^1_c$ by
\begin{eqnarr}\label{E:GFG}
{\mathcal A}f(x) &=&c(x) f'(x) +  \int_{(0,1)} r^{-1} f(rx) \bar k(x, \dd r) -{B}(x) f(x) \nonumber \\
&=&
c(x) f'(x) + {B}(x) \int_{\Pp} \left( \sum_{i=1}^{\infty}f(xp_i)-f(x)\right)\kappa(x, \dd \pp),
\end{eqnarr}
then comparing Lemma \ref{L1} above and Lemma 2.2 in \cite{BW} shows that the intensity measure $\mu_t$ of $\Zz_t$ solves \eqref{E:GFE} with $\mu_0=\delta_{x_0}$.
A fairly natural approach for establishing Lemma \ref{L1} would be to argue first that the intensity measure of $\Zz_t$ solves the growth-fragmentation equation for ${\mathcal A}$ given by \eqref{E:GFG} and then invoke Lemma 2.2 in \cite{BW}. This idea is easy to implement when the number of daughters after a fission event is bounded (for instance, when fissions are always binary); 
however, making  this analytic approach fully rigorous in the general case would be rather tedious, as the total number of individuals may explode in finite time and thus fission events accumulate. We rather follow a classical probabilistic approach and refer to the treatise by Del Moral \cite{Delmoral} and the lecture notes of Shi \cite{Shi} for background.

\begin{proof} %[Proof of Lemma \ref{L1}] 
We set $T_0=0$ and then write $T_1<T_2<\dotsb$ for the sequence of the jump times of the piecewise deterministic Markov process $X$.
We claim that for every generation $n\geq 0$,  there is the identity
 \begin{equation}\label{e:m21gen}
    \EE_{x_0} \left[ \sum_{|u|=n} f(Z^u_t) \Ind_{\{b^u \leq t  < d^u\}}  \right]
    = x_0\EE_{x_0}\left[ \Ind_{\{T_n \leq t  < T_{n+1}\}}\frac{f(X_t)}{X_t}\Ee_t \right].
  \end{equation}
The many-to-one formula of Lemma \ref{L1} then follows by summing over all generations.
  
  We shall now establish \eqref{e:m21gen} by iteration. 
The identity
\begin{equation}\label{e:Eexp}
\exp\left( \int_0^t \frac{c(x(s))}{x(s)}\dd s\right) = \frac{x(t)}{x(0)}
\end{equation}
for the solution to the flow velocity \eqref{e:ode}
 makes  \eqref{e:m21gen}  obvious for the generation $n=0$.

 Next, by considering the fission rates of Eve, we get that for every measurable function $g\from [0,\infty)\times [0,\infty)\to \RR_+$ with $g(t,0)=0$, we have
\begin{equation}
 \EE_{x_0}\left[ \sum_{i=1}^{\infty} g(b^i, Z^i_{b^i})\right]
 =\int_0^{\infty} \dd t {B}(x(t)) \exp\left(-\int_0^t {B}(x(s)) \dd s\right) \int_{\Pp} \kappa(x(t), \dd \pp)
 \sum_{i=1}^{\infty}g(t,x(t)p_i).
 \label{e:interm1}
\end{equation}
  We then write 
  $$ \sum_{i=1}^{\infty}g(t,x(t)p_i)=x(t) \sum_{i=1}^{\infty}p_i\frac{g(t,x(t)p_i)}{x(t)p_i},$$
 so that by comparing with the jump rates of $X$, we see that
 the right-hand side of \eqref{e:interm1} 
 equals
$$
 \EE_{x_0}\left[ \frac{g(T_1, X_{T_1})}{X_{T_1}}X_{T_1-}\right] = 
  x_0 \EE_{x_0}\left[ \frac{g(T_1, X_{T_1})}{X_{T_1}}\Ee_{T_1}\right],
$$
where the  identity stems from \eqref{e:Eexp}. Putting the pieces together, we have shown that
\begin{equation} \label{e:interm}
\EE_{x_0}\left[ \sum_{i=1}^{\infty} g(b^i, Z^i_{b^i})\right]= x_0 \EE_{x_0}\left[ \frac{g(T_1, X_{T_1})}{X_{T_1}}\Ee_{T_1}\right].
\end{equation}

We then assume that  \eqref{e:m21gen} holds for a given $n\geq 0$. Applying the branching property at the fission event of Eve, we get 
 $$   \EE_{x_0} \left[ \sum_{|u|=n+1} f(Z^u_t) \Ind_{\{b^u \leq t  < d^u\}}  \right]= 
 \EE_{x_0}\left[ \sum_{i=1}^{\infty} g(b^i, Z^i_{b^i})\right],
 $$
 with 
 $$g(s,y)= \EE_{y} \left[ \sum_{|u|=n} f(Z^u_t) \Ind_{\{b^u \leq t-s  < d^u\}}  \right] = y\EE_y\left[ \Ind_{\{T_n \leq t-s  < T_{n+1}\}}\frac{f(X_{t-s})}{X_{t-s}}\Ee_{t-s}
 \right]
  $$
 for $s\leq t$  and $g(s,y)=0$ otherwise. We conclude from the strong Markov property at the first jump time $T_1$ of $X$, 
the fact that  the functional $\Ee$ is multiplicative, and \eqref{e:interm},
 that the many-to-one formula \eqref{e:m21gen} holds for the generation $n+1$.
  By induction, \eqref{e:m21gen} holds for any $n$.
\end{proof}

In the final section of this work, we shall also need a version of Lemma \ref{L1} extended to the situation where, roughly speaking, individuals are frozen at
times which are observable from their individual trajectories.
Specifically, we define a \emph{simple stopping line} to be a functional $T$ on the space of piecewise continuous trajectories $z=(z_t)_{t\geq 0}$ and with values in  $[0,\infty]$, such that for every $t\geq 0$ and every trajectory $z$, if $T(z)\leq t$, then $T(z)=T(z')$ for any trajectory $z'$ that coincides with $z$ on the time-interval $[0,t]$. Typically, $T(z)$ may be the instant of the $j$-th jump of $z$, or the first entrance time $T(z)=\inf\{t>0: z_t\in A\}$ in some measurable set $A\subset (0,\infty)$. 
The notion of a simple stopping line is a particular case of the more general
stopping line introduced by Chauvin \cite{Cha-stop}. The restriction simplifies
the proofs somewhat, and will be sufficient for our applications later.

We next introduce  the notion of ancestral trajectories. Recall  from the preceding section the construction of the trajectory $(Z^u_t: t\in[b^u, d^u))$ for an individual $u=u_1\ldots u_n\in\Uu$.
The  sequence of prefixes $u^j=u_1 \ldots u_j$ for $j=1, \ldots, n$ forms the ancestral lineage of that individual. Note that, as customary for many branching models, the death-time of a mother always coincides with the birth-time of her children, so every individual $u$ alive at time $t>0$ (i.e. with $b^u\leq t < d^u$) has a unique ancestor alive at time $s\in[0,t)$, which is the unique prefix $u^j$ with $b^{u^j}\leq s < d^{u^j}$. We can thus define unambiguously  the mass at time $s$ of the unique ancestor of $u$ which is alive at that time, viz.\  $Z^u_s=Z^{u^j}_s$. This way, we extend $Z^u$ to $[0,d^u)$, and get the ancestral trajectory of the individual $u$.

For the sake of simplicity, for any simple stopping line $T$ and any trajectory $z$, we write $z_T=z_{T(z)}$, and define
the point process of individuals frozen at $T$ as
$$\Zz_T = \sum_{u\in\Uu} \Ind_{\{T(Z^u)\in[b^u, d^u)\}} \delta_{Z^u_T}.$$

\begin{lemma}\label{L1'} Let $T$ be a simple stopping line. 
    For every measurable $f\from (0,\infty)\to \RR_+$ and every $x_0>0$, we have
  \[
    \EE_{x_0} \left[ \langle \Zz_T, f\rangle \right]
    = x_0\EE_{x_0}\left[ \frac{f(X_T)}{X_T}\Ee_{T(X)}, T(X)<\infty \right].
  \]
\end{lemma}
\begin{proof} The proof is similar to that of Lemma \ref{L1}, and we use the same notation as there.
In particular, we write $x(\cdot)$ for the solution to the flow velocity \eqref{e:ode} started from $x(0)=x_0$, 
and set $T(x(\cdot))=t_0\in[0,\infty]$. By the definition of a simple stopping line, we have obviously that under $\PP_{x_0}$,
$T(Z^{\varnothing}) = t_0$ a.s.\ on the event $0\leq T(Z^{\varnothing})\leq d^{\varnothing}$, and 
also $T(X)=t_0$ a.s.\ on the event $0\leq T(X)\leq T_1$. 
Using \eqref{e:Eexp}, we then get
$$\EE \left[f(Z^{\varnothing}_T) \Ind_{\{b^{\varnothing} \leq T(Z^{\varnothing})  < d^{\varnothing}\}} \right] =x_0\EE_{x_0}\left[ \Ind_{\{0 \leq T  < T_{1}\}}\frac{f(X_T)}{X_T}\Ee_T \right].$$
Just as in the proof of Lemma \ref{L1}, it follows readily by induction 
that for every generation $n\geq 0$, there is the identity
$$
    \EE_{x_0} \left[ \sum_{|u|=n} f(Z^u_T) \Ind_{\{b^u \leq T(Z^u)  < d^u\}}  \right]
    = x_0\EE_{x_0}\left[ \Ind_{\{T_n \leq T  < T_{n+1}\}}\frac{f(X_T)}{X_T}\Ee_T \right],
$$
and we conclude the proof by summing over generations. 
\end{proof}

\section{Boundedness of the intrinsic martingale in \texorpdfstring{$L^2(\PP)$}{L\textasciicircum 2(P)}  }
 In order to apply results from \cite{BW,BeFK}, we shall now make some further fairly mild assumptions that will be enforced throughout the rest of this work. 
Specifically, we suppose henceforth that
\begin{equation}   \label{e:assump}
  \text{the Markov process $X$, with generator $\Gg$, is irreducible and aperiodic.}
\end{equation}
Although \eqref{e:assump} is expressed in terms of the Markov process $X$ rather than the characteristics of the growth-fragmentation process, it is easy to give some fairly general and simple conditions 
in terms of $c, B$ and $\kappa$ that guarantee \eqref{e:assump}; see notably Lemma 3.1 of \cite{BeFK} for a discussion of irreducibility. We further stress that aperiodicity should not be taken to granted if we do not assume the jump kernel $\bar k$ to be absolutely continuous.

\begin{remark} We mention that  a further assumption is made in \cite{BW,BeFK}, namely that the kernel $\bar k(x,\dd y)$ is absolutely continuous with respect to the Lebesgue measure, and that the function $(0,\infty) \ni x \mapsto \bar{k}(x,\cdot) \in L^1(0,\infty)$ is continuous. However, this is only needed in \cite{BW} to ensure some analytic properties (typically, the Feller property of the semigroup, or the connection with the eigenproblem \eqref{E:eigenp}), but had no role in the probabilistic arguments developed there. 
We can safely drop this assumption here, and apply results of \cite{BW,BeFK} for which it was irrelevant. 
\end{remark}

Following \cite{BW}, we introduce the Laplace transform
\[
  L_{x,y}(q) = \EE_x\bigl[ \e^{-qH(y)} \Ee_{H(y)} \Indic{H(y) < \infty} \bigr],\qquad q\in\RR,
\]
where $H(y) = \inf\{ t > 0: X_t = y \}$.
For any  $x_0>0$, the map $L_{x_0,x_0}\from \RR \to (0,\infty]$ is a convex non-increasing function
with  $\lim_{q\to \infty}L_{x_0,x_0}(q)=0$. We then {\em define} the {\em Malthus exponent} as
$$\lambda \coloneqq \inf\{q\in\RR: L_{x_0,x_0}(q)<1\}.$$
Recall that the value of $\lambda$ does not depend on the choice for $x_0$, and that although our definition of the Malthus exponent apparently differs from that in Section 5 of \cite{Jagers}, Proposition 3.3 of \cite{BW} strongly suggests that the two actually should yield the same quantity.

With this in place, we define the functions $\ell, h \from (0,\infty) \to (0,\infty)$ by
\[ \ell(x) = L_{x, x_0}(\lambda) \text{ and } h(x) = x \ell(x),  \] 
and may now state the main result of this section.

\begin{theorem} \label{T1}
Assume  
\begin{equation}\label{e:cc}
 \limsup _{x\to 0+} \frac{c(x)}{x}< \lambda \quad\text{and}\quad  \limsup _{x\to \infty} \frac{c(x)}{x}<\lambda.
 \end{equation}
 Then  for every $x>0$, the process
  \[ W_t = \e^{-\lambda t} \crochet{\Zz_t, h}, \qquad t\geq 0\]
  is a martingale bounded in  $L^2(\PP_x)$.
%  Moreover, its terminal value $W_{\infty}$ is positive $\PP_x$-a.s.
\end{theorem}

Before tackling the core of the proof of Theorem \ref{T1}, let us first recall some features proved in \cite{BW,BeFK} and their immediate consequences. From Section 3.5 in \cite{BeFK}, it is known that  \eqref{e:cc} ensures the existence of some $q<\lambda$ with $L_{x_0,x_0}(q)<\infty$. 
By continuity and non-increase of  the function $L_{x_0,x_0}$ on its domain, this guarantees that
\begin{equation}  \label{e:assump:malthus}
  L_{x_0 , x_0}(\lambda)=1.
\end{equation}
Theorem 4.4 in \cite{BW} then shows that $\e^{-\lambda t} \ell(X_t) \Ee_t$ is a $\PP_{x}$-martingale, 
and we can combine the many-to-one formula of Lemma \ref{L1} and the branching property of $\Zz$ to conclude that $W_t$ is indeed a $\PP_x$-martingale. We
therefore call $h$ a \emph{$\lambda$-harmonic function}; in this vein,
recall also from Corollary 4.5 and Lemma 4.6 in \cite{BW}  that
$h$ is an eigenfunction for the eigenvalue $\lambda$ of (an extension of) the growth-fragmentation operator ${\mathcal A}$ which has been defined in \eqref{E:GFG}.

We call $W=(W_t: t\geq 0)$ the \emph{intrinsic martingale}, as it bears a close connection to the process with the same name that has been defined in Section 5 of \cite{Jagers}. To explain this connection, it is convenient to  view the atomic measure $\e^{-\lambda t} \Zz_t$ as a weighted version of point measure $\Zz_t$, where  the weight of any individual at time $t$ is $\e^{-\lambda t}$. In this setting, $W_t$ is given by the integral of the $\lambda$-harmonic function $h$ with respect to the weighted atomic measure $\e^{-\lambda t} \Zz_t$. 
Next, consider for each $k\in \NN$, the simple stopping line $T_k$ at which a trajectory makes its $k$-th jump, and recall from the preceding section, that $\Zz_{T_k}$ then denotes the point measure obtained from $\Zz$ by freezing individuals at the instant when their ancestral trajectories jump for the $k$-th time. In other words, $\Zz_{T_k}$ is the point process that describes the position at birth of the individuals of the $k$-th generation. Just, as above, we further discount the weight assigned to each individual  at rate $\lambda$, so that the weight of an individual of the $k$-th generation which is born at time $b$ is $\e^{-\lambda b}$ (of course, individuals at the same generation are born at different times, and thus have different weights). The integral, say ${\mathcal W}_k$, of  the $\lambda$-harmonic function $h$ with respect to this  atomic measure, is precisely the intrinsic martingale as defined in \cite{Jagers}. Using more general stopping line techniques, one can check that the boundedness in $L^2$  of $(W_t: t\geq 0)$ can be transferred to $({\mathcal W}_k: k\in\NN)$. Details are left to the interested reader. 

 \begin{remark} Actually \eqref{e:assump:malthus}, which is a weaker assumption than \eqref{e:cc}, not only ensures that the process in continuous time $W_t=\e^{-\lambda t} \crochet{\Zz_t, h}$ is  a martingale, but also
 that the same holds for the process indexed by generations $({\mathcal W}_k: k\in\NN)$. Indeed, 
 from the very definition of the function $L_{x_0,x_0}$, \eqref{e:assump:malthus} states that the expected value under $\PP_{x_0}$
 of the nonnegative martingale $\e^{-\lambda t} \ell(X_t) \Ee_t$, evaluated at the first return time $H(x_0)$, equals $1$, and therefore the stopped martingale 
 $$\e^{-\lambda t\wedge H(x_0)} \ell(X_{t\wedge H(x_0)}) \Ee_{t\wedge H(x_0)}, \qquad t\geq 0$$ is uniformly integrable. Plainly, the first jump time of $X$, $T_1$ occurs before $H(x_0)$, and  the optional sampling theorem yields
 $$\EE_{x_0}[\e^{-\lambda T_1} \Ee_{T_1} \ell(X_{T_1})]=1.$$ 
 One concludes from the many-to-one formula of Lemma \ref{L1'} (or rather, an easy pathwise extension of it) 
 that $\EE_{x}[{\mathcal W}_1]=h(x)$ for all $x>0$, and the martingale property of ${\mathcal W}$ can  now be seen from the branching property.  \end{remark}

The rest of this section is devoted to the proof of Theorem \ref{T1}; in particular we assume henceforth that \eqref{e:cc} is fulfilled.

To start with, we recall from Lemma 4.6 of \cite{BW} that  the function $\ell$ is bounded and continuous, and 
as a consequence
\begin{equation}\label{e:bornh}
\sup_{y>0} h(y)/y = \sup_{y>0} \ell(y)=\| \ell \|_{\infty}<\infty.
\end{equation}   
Moreover $\ell$ and  $h$ are strictly positive, and thus bounded away from $0$ on compact subsets of $(0,\infty)$.  We shall use often these facts in the sequel.

The heart of the matter is thus to establish boundedness of $(W_t)_{t\ge 0}$
in $L^2(\PP_x)$, for which we follow the classical path based on  the probability tilting and spine decomposition; see e.g. \cite{BigKyp} and references therein. 
  
  For an arbitrary time $t>0$, one defines a probability measure $\tilde \PP_x$ on an augmented probability space by further distinguishing an individual $U_t$, called the spine, in such a way that 
  \[
    \tilde{\PP}_x[ \Lambda \cap \{ U_t = u \}]
    = h(x)^{-1}\e^{-\lambda t} \EE_x[  {h}(Z^u_t) \Ind_{\Lambda}\Ind_{\{b^u \leq t < d^u\}}]
  \]
  for $\Lambda \in \FF_t = \sigma( \Zz_t, s \le t)$ and $u\in \Uu$.
  The projection of $ \tilde{\PP}_x$ on $\FF_t$ is then absolutely continuous with respect to $\PP_x$ with density 
  $W_t/W_0$.
Recall that the martingale property of $W$ ensures the consistency of the definition of $\tilde \PP_x$. Precisely, 
 under  the conditional law $ \tilde{\PP}_x[\cdot \mid \FF_t]$, the spine is picked at random amongst the individuals alive at time $t$ according to an $h$-biased sampling, and the ancestor $U_s$ of $U_t$ at time $s\leq t$ serves as spine at time $s$.  
 
  In order to describe the dynamics of the mass of the spine $\tilde X_t=Z^{U_t}_t$ as time passes, we introduce first for every  $x>0$ 
$$w(x)=\int_{\Pp}\sum_{i=1}^{\infty} h(xp_i) \kappa(x, \dd \pp)$$
and  set
\begin{equation}\label{e:ratespine}
\tilde {B}(x)= \frac{w(x)}{h(x)}{B}(x)\ \text{and}\  \tilde \kappa(x, \dd \pp)= w(x)^{-1} \sum_{i=1}^{\infty} h(xp_i) \kappa(x, \dd \pp).
\end{equation}
In short, one readily checks that just as  $X$,  $\tilde X$  increases steadily and has only negative jumps. Its growth is driven by the flow velocity \eqref{e:ode}, and the total rate of negative jumps at location $x$ is  
$\tilde {B}(x)$, which is the total fission rate of the spine when its mass is $x$. Further, $\tilde \kappa(x, \dd \pp)$
gives  the distribution of the random mass partition resulting from the fission of the spine, given that the mass of the latter immediately before that fission event is $x$.  
At the fission event of the spine, a daughter is selected at random by $h$-biased sampling and becomes the new spine. We now gather some facts about the spine which will be useful later on.

\begin{lemma}\label{L20} 
Set $\tilde X_t=Z^{U_t}_t$ for the mass of the spine at time $t\geq 0$. 
\begin{enumerate}
\item The process 
 $\tilde X$ is Markovian and  exponentially point recurrent, in the sense that 
if we write $\tilde H(y) =\inf\{t>0: \tilde X_t=y\}$ for the first hitting time of $y>0$ by $\tilde X$, then there exists 
$\varepsilon>0$ such that
$\EE_x[\exp(\varepsilon \tilde H(y)] <\infty$.

\item The following many-to one formula holds: for every nonnegative measurable function $f$ on $(0,\infty)$, we have
\begin{equation}\label{e:mtoY}
  \EE_x[\crochet{\Zz_t, f}]=\e^{\lambda t} h(x)\tilde \EE_x[f(\tilde X_t)/h(\tilde X_t)].
\end{equation}

\item  Any function $g\from (0,\infty)\to \RR$ such that $g\ell$ is continuously differentiable belongs to the domain
of its extended infinitesimal generator $\tilde {\mathcal G}$ and 
\begin{equation}\label{E:genY}
    \tilde {\mathcal G}g(x) 
      = \frac{1}{\ell(x)} {\mathcal G}(g\ell)(x) +(c(x)/x- \lambda) g(x),
\end{equation}
in the sense that the process 
$$
      g(\tilde X_t)-\int_0^t \tilde {\mathcal G}g(\tilde X_s) \, \dd s
$$ 
is   a  $\tilde \PP_x$-local martingale for every $x>0$.
\end{enumerate}

\end{lemma}
\begin{proof}  It follows immediately from the definition of the spine and the many-to-one formula of Lemma \ref{L1} that for every $t\geq 0$, the law of $\tilde X_t$ under $\tilde \PP_x$ is absolutely continuous with respect to that of $X_t$ under $\PP_x$, with density $\e^{-\lambda t}  \Ee_t \ell(X_t)/\ell(X_0)$.
Recall that the latter is a $\PP_x$-martingale, which  served in Section 5 of \cite{BW} to construct a point-recurrent Markov $Y=(Y_t,  t\geq 0)$ by probability tilting. Hence $Y$ has  the same one-dimensional marginals as $\tilde X$,
and since the two processes are Markovian, they have the same law. 

The claim that $\tilde X$ (that is equivalently, $Y$)  is exponentially point recurrent, stems from the proof of Theorem 2 of \cite{BeFK} and Lemma 5.2(iii) of \cite{BW}. The many-to-one formula \eqref{e:mtoY} merely rephrases the very definition of the spine. 
Finally, the third claim about the infinitesimal  generator follows from Lemma 5.1 in \cite{BW}.
\end{proof}

\begin{remark}
The description of the dynamics governing the evolution of the spine entails that
its infinitesimal generator can also be expressed by
\begin{eqnarray}\label{E:gentilde}
     \tilde {\mathcal G}f(x) &=& c(x) f'(x) + \frac{B(x)}{h(x)} \int_{\Pp} \left(\sum_{i=1}^{\infty}h(xp_i)f(xp_i)-f(x)\right)  \kappa(x, \dd \pp),
     \end{eqnarray}
say for any $f\in{\mathcal C}^1_c$. The agreement between
\eqref{E:genY} and \eqref{E:gentilde} can be seen from the identity
$\mathcal{G}\ell(x) = (\lambda - c(x)/x)\ell(x)$, which is proved in
Corollary 4.5(i) of \cite{BW}.
\end{remark}
 We readily deduce from Lemma \ref{L20} that the intensity measure of the growth-fragmentation satisfies the Malthusian behavior \eqref{E:Malthus} uniformly on compact sets.

\begin{corollary}\label{C1}  For every compact set $K\subset (0,\infty)$  and every continuous function $f$ with $\|f/h\|_{\infty}<\infty$, we have
$$\lim_{t\to \infty} \e^{-\lambda t}\EE_x[\crochet{ \Zz_t,f}]= h(x) \crochet{\nu,f} \qquad \text{uniformly for }x\in K,$$
where the asymptotic profile  is given by $\nu=h^{-1}\pi$, with $\pi$ 
the unique stationary law of the spine process $\tilde X$.
\end{corollary}
\begin{proof} Suppose $K\subset [b,b']$ for some $0<b<b'<\infty$, and fix $ \varepsilon <\in (0,1)$.
For every $0<x<y$, let $s(x,y)$ denote the instant when the flow velocity \eqref{e:ode} started from $x$ reaches $y$. 
  
Since the total jump rate $\tilde B$ of $\tilde X$ remains bounded on $K$ and the growth rate $c$ is bounded away from $0$ on $K$, we can find a finite sequence $b=x_1< x_2 < \ldots < x_j=b'$ 
such that for every $i=1, \ldots, j-1$ and every $x\in(x_i, x_{i+1})$:
$$s(x,x_{i+1})<1 \text{ and } s(x_i,x) < 1,$$
as well as
$$\tilde \PP_{x}(\tilde H(x_{i+1})=s(x,x_{i+1})) >1-\varepsilon \quad \text{and}\quad  \tilde \PP_{x_i}(\tilde H(x)=s(x_i,x)) >1-\varepsilon. $$
An immediate application of the simple Markov property now shows that for every $i=1, \ldots, j-1$, every $x\in(x_i, x_{i+1})$, every $t\geq 1$, and every nonnegative measurable function $g$, the following bounds hold
\[ 
  (1-\varepsilon) \tilde \EE_{x_{i+1}}(g(\tilde X_{t-s(x,x_{i+1})}))
  \le
   \tilde \EE_{x}(g(\tilde X_{t}))
  \le
  (1-\varepsilon)^{-1} \tilde \EE_{x_i}(g(\tilde X_{t+s(x_i,x)}))
\]

On the other hand, we know that  $\tilde X$ is irreducible, aperiodic and ergodic, with stationary law $\pi$  (recall Lemma \ref{L20}(i)). Since $\varepsilon$ can be chosen arbitrarily small, it follows from above  that $\tilde X$ is uniformly ergodic on $K$, in the sense that for every continuous and bounded function $g$, 
$$\lim_{t\to \infty} \tilde \EE_{x}(g(\tilde X_{t}))= \crochet{\pi, g} \qquad \text{uniformly for }x\in K.$$
We conclude the proof with an appeal to the many-to-one formula  of Lemma \ref{L20} (ii), taking $g=f/h$.
\end{proof} 

 The next lemma is a cornerstone of the proof of Theorem \ref{T1}.
\begin{lemma}\label{L2} 
We have:
\begin{enumerate}
\item There exists $a<\infty$ such that, for all $x>0$ and $t\geq 0$, 
$$\tilde \EE_x[1/\ell(\tilde X_t)]\leq  at + 1/\ell(x).$$

\item There exists some $\lambda'<\lambda$ such that, for all $x>0$,
$$\lim_{t\to \infty} \e^{-\lambda' t} \tilde X_t=0\quad \text{ in }
L^{\infty}(\tilde \PP_x).$$
\end{enumerate} 
\end{lemma}

\begin{proof}  (i) We apply Lemma \ref{L20}(iii)
to $g=1/\ell$, with 
   $$\tilde {\mathcal G}\left(\frac{1}{\ell}\right) (x)= (c(x)/x-\lambda)/\ell(x).$$
Our assumption \eqref{e:cc} ensures that the right-hand side above is negative for all $x$ aside from some compact subset of $(0,\infty)$. 
Taking $a=\sup_{x>0} \tilde {\mathcal G}\left(1/{\ell}\right) <\infty$, 
we deduce from Lemma \ref{L20}(iii) by optional sampling that
$\tilde \EE_x[1/\ell(\tilde X_t)]-a t\leq 1/\ell(x)$, which entails our claim.  

(ii)  Recall from the description of the dynamics of the spine before the statement  that $\tilde X$ increases continuously with  velocity   $c$ and has only negative jumps. As a consequence, $\tilde X$ is bounded from above by the solution to the flow velocity
   \eqref{e:ode}. One readily deduces that $\lim_{t\to \infty} \e^{-\lambda' t}x(t)=0$ for every $\lambda' >\limsup_{x\to \infty} c(x)/x$, and since $\limsup_{x\to \infty} c(x)/x<\lambda$ according to our standing assumption \eqref{e:cc}, this establishes our claim. 
\end{proof}

     We now have all the ingredients needed to prove Theorem \ref{T1}
     
        \begin{proof}[Proof of Theorem \ref{T1}] 
 Since the projection on ${\mathcal F}_t$ of $\tilde \PP_x$ is absolutely continuous with density $W_t/W_0$,  
 the process $W$ is bounded in $L^2(\PP_x)$  if and only if $
\sup_{t\geq 0} \tilde \EE_x [W_t] <\infty$.  We already know from Lemma \ref{L2}(ii) that $\sup_{t\geq 0} \tilde \EE_x [\e^{-\lambda t} \tilde X_t] <\infty$,
and we are thus left with checking that 
  \begin{equation}\label{e:goal}
\sup_{t\geq 0} \tilde \EE_x [W'_t] <\infty,
\end{equation} 
where 
$$W'_t=W_t-\e^{-\lambda t} \tilde X_t.$$

In this direction, it is well-known and easily checked that 
  the law of $\Zz$ under the new probability measure  $\tilde{\PP}_x$ can be constructed by the following procedure, known as the spine decomposition. After each fission event of the spine, all the daughters except the new spine
  start independent growth-fragmentation processes following the genuine dynamics of $\Zz$ under $\PP$. 
This  spine decomposition 
 enables us to estimate the conditional expectation of   $W'_t$ under $\tilde \PP_x$, given the spine and its sibling. At each time, say $s>0$, at which a fission occurs for the spine, we write
    $\pp(s)=(p_1(s), \ldots)$ for the resulting mass partition, and $I(s)$ for the index of the daughter spine. Combining this with the fact that $(W_s: s\geq 0)$ is a $\PP_y$-martingale for all $y>0$ entails the identity
     $$\tilde \EE_x\left[W'_t\mid (\tilde X_s, \pp(s), I(s))_{s\geq 0}\right] = 
     \sum_{s\in \tilde F, s\leq t} \e^{-\lambda s} \sum_{i\neq I(s)} h(\tilde X_{s-} p_i(s))\,,$$
where $\tilde F$ denotes the set of fission times of the spine. Note from \eqref{e:bornh}
that $\sum h(xp_i) \leq x \| \ell\|_{\infty}$ for every $x>0$ and every mass-partition ${\mathbf p}=(p_1, \ldots)$, so the right-hand side is bounded from above by 
$$ \| \ell \|_{\infty} \sum_{s\in \tilde F, s\leq t} \e^{-\lambda s} \tilde X_{s-},$$
   and to prove \eqref{e:goal}, we now only need to check that 
   $$\tilde \EE_x \left[ \sum_{s\in \tilde F} \e^{-\lambda s} \tilde X_{s-}\right] <\infty . $$
   
   Recall from \eqref{e:ratespine} that $\tilde {B} = w{B}/h$ describes the fission rate of the spine, and observe from
   \eqref{e:bornh} that $w(x)\leq \| \ell\|_{\infty} x$, so that
$$\tilde {B}(x) \leq  \| \ell\|_{\infty}  \| {B} \|_{\infty}
\frac{1}{\ell(x)} 
\qquad \text{for all }x>0.$$
   This entails that
   $$\tilde \EE_x \left[ \sum_{s\in \tilde F} \e^{-\lambda s} \tilde X_{s-}\right] \leq  
     \| \ell\|_{\infty}  \| {B} \|_{\infty}\tilde \EE_x\left[ \int_0^{\infty} \e^{-\lambda s} \frac{\tilde X_s}{\ell(\tilde X_s)}\dd s\right]. $$
We now see that the expectation in the right-hand side is indeed finite by  writing first
   $$\int_0^{\infty} \e^{-\lambda s} \frac{\tilde X_s}{\ell(\tilde X_s)}\dd s
   = \int_0^{\infty}  \frac{1}{\ell(\tilde X_s)} \cdot \e^{-\lambda' s}\tilde X_s
   \cdot \e^{-(\lambda-\lambda')s} \, \dd s
  $$
and then applying  Lemma \ref{L2}.    
    \end{proof}

\section{Strong Malthusian behavior}

We assume again throughout this section that the assumption \eqref{e:cc} is fulfilled.
We will prove Theorem \ref{t:main}:
the strong Malthusian behavior \eqref{E:MalthusF} then holds.
%% I have removed this theorem as it now appears in the introduction.
% \begin{theorem} \label{T2}
%  Assume \eqref{e:cc}.
%  Then for every $x>0$ and every continuous function $f$ with $\|f/h\|_{\infty}<\infty$, we have
%    \[ \lim_{t\to \infty} \e^{-\lambda t} \crochet{\Zz_t,f} = 
%    \ip{\nu}{f} W_\infty, \qquad \text{in }L^1(\PP_x)
%  \]
%  where $W_\infty$ is the terminal value of
%  the intrinsic martingale $W$.
%\end{theorem}

The proof relies on a couple of technical lemmas. Recall from Section 3 the notation $Z^u\colon  [0, d^u)\to(0,\infty)$ for the ancestral trajectory of the individual $u$, and agree for the sake of simplicity that $f(Z^u_t)=0$ for every function $f$ whenever $t\geq d^u$. 

The first lemma states  a simple tightness result.
\begin{lemma} \label{L:B} For every $x>0$ and $\varepsilon >0$, there exists a compact $K\subset (0,\infty)$ such that for all $t\geq 0$:
$$ \e^{-\lambda t}  \EE_{x}\left[\sum_{u\in\Uu: b^u\leq t < d^u }  h(Z^u_t) \Indic{Z^u_t\not\in K} 
\right]< \varepsilon.$$
\end{lemma}
\begin{proof} From the very definition of the spine $\tilde X$, there is the identity
$$ \e^{-\lambda t}  \EE_{x}\left[ \sum_{u\in\Uu: b^u\leq t < d^u }  h(Z^u_t) \Indic{Z^u_t\not\in K} 
\right] = h(x) \tilde \PP_{x}\left[ \tilde X_t\not\in K\right].$$
Recall from Lemma \ref{L20}(i) that $\tilde X$ is positive recurrent; as a consequence the family of its one-dimensional marginals under $\tilde \PP_{x}$ is tight, which entails our claim. 
\end{proof}
The second lemma reinforces the boundedness in $L^2$ of  the intrinsic martingale, cf.\  Theorem \ref{T1}.  

\begin{lemma} \label{L:C} For every compact set $K\subset(0,\infty)$, we have
$$ \sup_{x\in K} \sup_{t\geq 0} \EE_x\left[ W^2_t \right] < \infty.$$
\end{lemma}
\begin{proof} We may assume that $K=[b,b']$ is a bounded interval. For any $x\in(b,b']$, we write  $s(x)$ for the time when the flow velocity \eqref{e:ode}  started from $b$ reaches $x$. We work under 
$\PP_{b}$ and consider the event $\Lambda_x$ that the Eve individual  hits $x$ before a fission event occurs. 

We have on the one hand, that the law of $\Zz_{s(x)+t}$ conditionally on 
$ \Lambda_x$ is the same as that of $\Zz_t$ under $\PP_x$. 
In particular, the law of $W_t$ under $\PP_x$ is the same as that of $\e^{\lambda s(x)} W_{s(x)+t}$ under 
$\PP_b[\cdot \mid \Lambda_x]$, and thus
$$\sup_{t\geq 0} \EE_x\left[ W^2_t \right] \leq \frac{\e^{\lambda s(x)}}{\PP_b[\Lambda_x]}  \EE_b\left[ W^2_{\infty} \right].$$

On the other hand, for every $x\in(b,b']$, we have $s(x)\leq s(b')<\infty$ and 
  $$ \PP_b[\Lambda_x]\geq \PP_b[\Lambda_{b'}]=\exp\left(-\int_{b}^{b'} \frac{{B}(y)}{c(y)}\dd y\right)>0,$$
   and  our claim is proven. 
\end{proof}
We have now all the ingredients to prove Theorem \ref{t:main}.

\begin{proofthm}{Proof of Theorem \ref{t:main}}
We suppose that $0\leq f \leq h$, which induces of course no loss of generality. Our aim is to check that 
$\e^{-\lambda(t+s)} \crochet{\Zz_{t+s},f} $ is arbitrarily close to $\crochet{\nu,f}W_t $ in $L^1(\PP_x)$ when $s$ and $t$ are sufficiently  large. In this direction, recall that  ${\mathcal F}_t$ denotes  the natural filtration generated by $\Zz_t$ and use
 the branching property at time $t$ to express the former quantity as  
 $$ \e^{-\lambda (t+s)}\crochet{\Zz_{t+s},f}= 
 \sum_{u\in\Uu: b^u\leq t < d^u}\e^{-\lambda t} h(Z^u_t) \cdot \frac{1}{h(Z^u_t)}\e^{-\lambda s}\crochet{\Zzu_{s},f},
$$
where conditionally on ${\mathcal F}_t$, the processes $\Zzu$ are independent versions of the growth-fragmentation $\Zz$ started from $Z^u_t$.

Fix $\varepsilon>0$.
To start with, we choose a compact subset $K\subset (0,\infty)$
as in Lemma \ref{L:B}, and restrict the sum above to individuals $u$ with $Z^u_t\not \in K$.
Observe first that, since $\crochet{\Zzu_s, f}\leq \crochet{\Zzu_s, h}$ and $h$ is $\lambda$-harmonic, taking the conditional expectation given ${\mathcal F}_t$ yields
\begin{multline*}
\EE_x\!\left[  \sum_{\substack{u\in\Uu, \\ b^u\leq t < d^u}}\e^{-\lambda t} h(Z^u_t) \Indic{Z^u_t\not \in K}\cdot \frac{1}{h(Z^u_t)}\e^{-\lambda s}\crochet{\Zzu_{s},f}\right]
 \leq \EE_x\!\left[  \sum_{\substack{u\in\Uu, \\ b^u\leq t < d^u}}\e^{-\lambda t} h(Z^u_t) \Indic{Z^u_t\not \in K}\right].
\end{multline*}
From the very choice of $K$, there is the bound
\begin{equation}\label{E:P1}
\EE_x\left[  \sum_{u\in\Uu: b^u\leq t < d^u}\e^{-\lambda t} h(Z^u_t) \Indic{Z^u_t\not \in K}\cdot \frac{1}{h(Z^u_t)}\e^{-\lambda s}\crochet{\Zzu_{s},f}\right]\leq \varepsilon.
\end{equation}

Next, recall from Lemma \ref{L:C} that 
$$C(K)\coloneqq 
 \sup_{y\in K} \sup_{s\geq 0} \EE_y\left[ W_s^2 \right] < \infty,
 $$
 and consider 
  $$A(u,t,s)=\frac{1}{h(Z^u_t)}\e^{-\lambda s}\crochet{\Zzu_{s},f} $$
 together with  its conditional expectation given ${\mathcal F}_t$ 
 $$a(u,t,s)=\EE_x[A(u,t,s)\mid {\mathcal F}_t].$$
 Again, since $0\leq f \leq h$, 
 for every $u$ with $Z^u_t\in K$, we have 
$$\EE_x[(A(u,t,s)-a(u,t,s))^2\mid  {\mathcal F}_t] \leq 4 C(K). 
$$

Since conditionally on ${\mathcal F}_t$, the variables $A(u,t,s)-a(u,t,s)$ for $u\in{\mathcal U}$ are independent and centered, there is the identity 
\begin{multline*}
  \EE_x\left[  \left| \sum_{u\in\Uu: b^u\leq t < d^u}\e^{-\lambda t} h(Z^u_t) \Indic{Z^u_t \in K}\cdot (A(u,t,s)-a(u,t,s))\right |^2
 \right]\\
  = \EE_x\left[  \sum_{u\in\Uu: b^u\leq t < d^u}\e^{-2\lambda t} h^2(Z^u_t) \Indic{Z^u_t \in K}\cdot \EE_x\bigl[(A(u,t,s)-a(u,t,s))^2\bigm\vert  {\mathcal F}_t\bigr]
  \right],
\end{multline*} 
 and we deduce from above that this quantity is bounded from above by 
 $$4 C(K)\e^{-\lambda t} h(x) \max_{y\in K} h(y).$$
 This upper-bound  tends to $0$ as $t\to \infty$, and it thus holds  that 
$$\EE_x\left[  \left| \sum_{u\in\Uu: b^u\leq t < d^u}\e^{-\lambda t} h(Z^u_t) \Indic{Z^u_t \in K}\cdot (A(u,t,s)-a(u,t,s))\right |
 \right]< \varepsilon$$
  for all $t$ sufficiently large. 

On the other hand, writing $y=Z^u_t$, we have from the branching property
$$a(u,t,s)=  \frac{1}{h(y)} \e^{-\lambda s} \EE_{y} \left[\crochet{\Zz_{s},f}\right],$$
and Corollary \ref{C1} entails that for all $s$ sufficiently large, 
$|a(u,t,s)-\crochet{\nu,f}|\leq \varepsilon$ for all individuals $u$ with $Z^u_t\in K$. Using the bound \eqref{E:P1} with $h$ in place of $f$ for individuals $u$ with $Z^u_t\not \in K$ and
putting the pieces together, we have shown that for all $s, t$ sufficiently large,
$$\EE_x\left[ \left|    \e^{-\lambda (t+s)}\crochet{\Zz_{t+s},f} - \crochet{\nu,f} W_t
\right | \right] \leq (2+h(x))\varepsilon,
$$
which completes the proof. 
\end{proofthm}

\section{Explicit conditions for the strong Malthusian behavior}

The key condition for strong Malthusian behavior, \eqref{e:cc}, is given in terms of the Malthus exponent $\lambda$, which is  not known explicitly in general. In this final section, we discuss explicit criteria in terms of the characteristics of $\Zz$ ensuring that $\lambda > 0$, so that condition \eqref{e:cc} then immediately follows from the simpler requirement 
$$\lim_{x\to 0+}c(x)/x = \lim_{x\to \infty}c(x)/x =0.$$
In this direction, we recall first that Theorem 1 in \cite{DoumGab} already gives sufficient conditions for the strict positivity of the leading eigenvalue in the eigenproblem \eqref{E:eigenp}. More recently, it has been pointed out in Proposition 3.4(ii) of \cite{BeFK} that if the Markov process $X$ is recurrent, and the uninteresting case when $c(x)=ax$ is a linear function excluded,
then $\lambda > \inf_{x>0}c(x)/x$.
(If $c(x) = ax$, then $\lambda=a$, $h(x)=x$, and one readily checks that the martingale $W$ is actually constant.)
It was further argued in Section 3.6  in \cite{BeFK}  that sufficient conditions warranting recurrence for $X$ are easy to formulate. 
For instance, it suffices that there exist some $q_{\infty}>0$ and $x_{\infty}>0$ such that
$$
q_{\infty} c(x)/x + \int_{(0,1)} (r^{q_{\infty}}-1)\bar k(x,\dd r) \leq 0 \qquad\text{ for all }x\geq x_{\infty},
$$
and also some $q_0>0$ and $x_0>0$ such that 
$$-q_0 c(x)/x + \int_{(0,1)} (r^{-q_0}-1) \bar k (x,\dd r) \leq 0 \qquad\text{ for all }x\leq x_0.
$$
% (Note that $B$ was assumed to be bounded in \cite{BeFK}.)

In this section, we shall show that a somewhat weaker condition actually suffices. For the sake of simplicity, we focus on the situation when fissions are binary (which puts $\Zz$ in the class of Markovian growth-fragmentations defined in \cite{BeGF}). However, it is immediate to adapt the argument to the general case.

Assume that, for all $x>0$,  $\kappa(x, \dd \pp)$ is supported by the set of binary  mass  partitions $\pp=(1-r,r, 0, \ldots)$ with $r\in(0,1/2]$.  It is then more convenient to represent the fission kernel $\kappa$ by a probability kernel $\varrho(x,\dd r)$ on $(0,1/2]$,  such that for all functionals $g\geq 0$ on $\Pp$, 
$$\int_{\Pp} g(\pp) \kappa(x, \dd \pp) = \int_{(0,1/2]} g(1-r,r, 0, \ldots) \varrho(x,\dd r).$$
In particular, there is the identity
$$\int_{(0,1)} f(r) \bar k (x,\dd r)= B(x)\int_{(0,1/2]} ((1-r) f(1-r)+ r f(r)) \varrho(x,\dd r).$$

\begin{proposition}\label{P1} In the notation above, assume that
there exist $q_{\infty}, x_{\infty}>0$ such that
% $x^{-q_\infty} B(x) \to 0$ as $x\to\infty$ and that
$$
q_{\infty} c(x)/x + B(x) \int_{(0,1/2]} (r^{q_{\infty}}-1) \varrho (x, \dd r) \leq 0 \qquad\text{ for all }x\geq x_{\infty},
$$
and $q_0, x_0>0$ such that
$$-q_0 c(x)/x + B(x) \int_{(0,1/2]} ((1-r)^{-q_0}-1) \varrho(x,\dd r) \leq 0 \qquad\text{ for all }x\leq x_0.
$$
Then, the Malthus exponent $\lambda$ is positive. 
\end{proposition}
\begin{proof}
Let $a \in (x_0,x_\infty)$.
By the definition of the Malthus exponent in Section 4 and the right-continuity of the function $L_{a,a}$, we see that $\lambda>0$ if and only if $L_{a,a}(0)\in(1,\infty]$. We thus have to check that
$$\EE_a\left[\Ee_{H(a)}, H(a)<\infty\right]>1,$$ 
that is, thanks to the many-to-one formula of Lemma \ref{L1'}, that 
$$\EE_a[\crochet{\Zz_{H(a)}, \mathbf{1}}]>1,$$
where $\mathbf{1}$ is the constant function with value $1$.
{\em A fortiori}, it suffices to  check that $\crochet{\Zz_{H(a)}, \mathbf{1}}\geq 1$ $\PP_a$-a.s., and that this inequality is strict with positive $\PP_a$-probability. In words, we freeze individuals 
 at their first return time to $a$; it is easy to construct an event with positive probability on which there are two or more frozen individuals, so we only need to verify that we get at least one frozen individual $\PP_a$-a.s.

In this direction, we focus on a specific ancestral trajectory, say  $X^*$, which is defined as follows. Recall that any trajectory is driven by the flow velocity \eqref{e:ode} between consecutive times of downward jumps, so we only need to explain how we select daughters at fission events. When a fission event occurs at a time $t$
with $X^*_{t-}=x^*>a$, producing two daughters, say $rx^*$ and $(1-r)x^*$ for some $r\in(0,1/2]$, then we choose the smallest daughter, i.e. $X^*_t=rx^*$, whereas if $x^*<a$ then we choose the largest daughter, i.e. $X^*_t=(1-r)x^*$.
The process $X^*$ is then Markovian with infinitesimal generator
$${\mathcal G}^*f(x)=c(x) f'(x) +B(x) \int_{(0,1/2]} (\Indic{x>a} f(rx) + \Indic{x<a}f((1-r)x) -f(x))\varrho(x,\dd r).$$

We now outline a proof that $X^*$ is point-recurrent.
Since the process has only negative jumps,
it is sufficient to show that $X^*_t \to \infty$ and $X^*_t \to 0$,
as $t\to\infty$, are both impossible.
For the former, consider starting the process at $x>x_\infty$
and killing it upon passage below $x_\infty$. Denote this process by
$X^\circledast$ and its generator by
$${\mathcal G}^\circledast f(x)=c(x) f'(x) +B(x) \int_{(0,1/2]} (\Indic{rx > x_\infty} f(rx) -f(x))\varrho(x,\dd r).$$
(The dependence on $a$ in the integral
vanishes since $x_\infty>a$.)
Now, let $V(x) = x^{q_\infty}$, for $x\ge x_\infty$.
The conditions in the statement imply that
$\mathcal{G}^\circledast V \le 0$, so
$V(X^\circledast)$ is a supermartingale. This ensures that
$X^\circledast$ cannot converge to $+\infty$, and indeed
the same for $X^*$ itself.
To show $X^*$ cannot converge to $0$, we start it at $x<x_0$
and kill it upon passing above $x_0$, and follow the same
argument with $V(x) = x^{-q_0}$.

To conclude, we have shown that
that $X^*$ is point-recurrent, and therefore $\PP_a$-almost surely hits $a$.
This shows that
$\PP_a[\crochet{\Zz_{H(a)}, \mathbf{1}}\geq 1]=1$, and completes the proof.
\end{proof}

We remark that a similar argument was carried out in Section 3.6 of \cite{BeFK},
using instead the Markov process $X$. The Markov process $X$
can be selected from the process $\mathbf{Z}$ by making a size-biased
pick from the offspring at each branching event; that is, from offspring
of sizes $rx$ and $(1-r)x$, following the former with probability $r$
and the latter with probability $1-r$. On the other hand, in the
process $X^*$ in the proof above, we pick from the offspring more carefully
in order to follow a line of descent which is more likely to stay close
to the point $a$. This accounts for the improvement in conditions
between \cite{BeFK} and this work.

\bibliography{Intrins}
 \eject

\end{document}